\documentclass[11pt, letter paper, oneside]{amsart}
\pdfoutput=1
\usepackage[T1]{fontenc}
\usepackage{lmodern}
\usepackage{yhmath}
\usepackage{amscd, amssymb, amsmath, amsthm, mathtools}
\usepackage{enumerate, latexsym, mathrsfs}
\usepackage{xy}
\usepackage{bm}
\usepackage{tikz-cd}
\usepackage{enumitem}

\usepackage{graphics,graphicx,psfrag, mathrsfs}
\usepackage{caption}
\graphicspath{{./img/}}

\usepackage{geometry}

\usepackage[
		bookmarks=true, bookmarksopen=true,%
    bookmarksdepth=3,bookmarksopenlevel=2,%
    colorlinks=true,%
    linkcolor=blue,%
    citecolor=blue]{hyperref}

\newtheorem{theorem}{Theorem}[section]
\newtheorem{lemma}[theorem]{Lemma}
\newtheorem{proposition}[theorem]{Proposition}

\newtheorem{corollary}[theorem]{Corollary}

\newtheorem*{fact}{Fact}

\theoremstyle{definition}
\newtheorem{definition}[theorem]{Definition}

\newtheorem{example}[theorem]{Example}

\newtheorem{remark}[theorem]{Remark}

\numberwithin{equation}{section}

\newcommand{\Real}{{\mathbb{R}}}

\newcommand{\Integer}{{\mathbb{Z}}}

\newcommand{\Hyperb}{{\mathbb{H}}}
\newcommand{\Sph}{{\mathbb{S}}}

\def\HC{{(\hyperref[hyperbolicity condition]{HC})}}
\def\RC{{(\hyperref[rigidity condition]{RC})}}

\DeclareMathOperator{\MCG}{\mathrm{MCG}}
\DeclareMathOperator{\Int}{\mathrm{int}}

\newtheorem*{thm:peripheralDehnfill}{Theorem~\ref{thm:peripheralDehnfill}}
\newtheorem*{thm:closedset}{Theorem~\ref{thm:closedset}}
\newtheorem*{thm:example}{Theorem~\ref{cor:example}}
\newtheorem*{cor:reducedcase}{Corollary~\ref{cor:reducedcase}}
\newtheorem*{thm:secondmain}{Theorem~\ref{thm:secondmain}}
\newtheorem*{thm:hyperbolicity}{Theorem~\ref{thm:hyperbolicity}}

\title[]{Profinite rigidity and geometric convergence}

\author[Yu Huang]{Yu Huang}
\address{Beijing International Center for Mathematical Research, Peking University\\
				Beijing 100871, China P.R.}
\email{hytopaz@stu.pku.edu.cn}

\begin{document}
\begin{abstract}
  In this paper, we prove that profinitely rigid finite-volume hyperbolic manifolds form a closed set under geometric topology.
  This observation implies the profinite rigidity of a large family of cusped hyperbolic manifolds via bubble-drilling construction.
  The core of the proof is a strong criterion that is used to verify when bubble-drilled manifolds are hyperbolic.
  This family includes many link complements, such as the Whitehead link complement and the Borromean ring complement.
\end{abstract}

\maketitle


\section{Introduction}\label{sec:intro}

A pair of finitely generated groups, \(G_{1}\) and \(G_{2}\), are said to be \textit{profinitely isomorphic} if they have isomorphic profinite completions, written as $\widehat{G_1}\cong \widehat{G_2}$. It is known that two finitely generated groups have isomorphic profinite completions if and only if they have the same collection of finite quotient groups.
In~\cite{Gr70a}, Grothendieck questioned whether finitely generated residually finite groups are determined by its profinite completions.

It is natural to formulate this question in the setting of 3-manifolds since 3-manifolds are largely determined by their fundamental groups.
Let $\mathfrak{M}$ denote the class of all compact orientable 3-manifolds without boundary spheres. A manifold $M\in \mathfrak{M}$ is called {\em profinitely rigid in $\mathfrak{M}$} if any manifold in \(\mathfrak{M}\) whose fundamental group is profinitely isomorphic to \(\pi_{1}M\) is homeomorphic to \(M\). In this paper,
all manifolds are assumed to be orientable and connected unless explicitly stated otherwise.

Profinite completion of the fundamental group reflects the geometric decomposition of 3-manifolds.
More specifically, profinite completion of the fundamental group determines whether a compact, orientable 3-manifold with possibly non-empty toral boundaries is geometric in the sense of Thurston, and, if so, identifies its geometry type, as established by a series of works by Wilton-Zalesskii~\cite{WZ17a,WZ17b,WZ19a}.
In fact, profinite classification in seven of the eight geometries proposed by Thurston has been well understood, except for the most mysterious hyperbolic case.
For instance, closed 3-manifolds admitting $\mathbb{S}^2\times \mathbb{E}^1$, $\mathbb{E}^3$, $Nil$ and $\widetilde{\mathrm{SL}(2,\mathbb{R})}$ geometry have been proven to be profinitely rigid in $\mathfrak M$ by Wilkes~\cite{Wi17a}, while 3-manifolds admitting $Sol$ or $\mathbb{H}^2\times \mathbb{E}^1$ geometries may not be profinitely rigid according to \cite{Fu13a} and \cite{He14a}, respectively.
We refer the reader to Reid's survey~\cite{Re18a} for more history about profinite rigidity in 3-manifolds.

In the hyperbolic case, Liu~\cite{Li23a} first showed the profinite almost rigidity of finite-volume hyperbolic 3-manifolds,
namely, their homeomorphism types are determined, up to finitely many possibilities, by the profinite completion of fundamental groups.

In this paper, we construct numerous cusped hyperbolic 3-manifolds that are profinitely rigid in \(\mathfrak{M}\) through the bubble-drilling construction.
These examples include many link complements in \(S^3\).
We first formulate some elegant and insightful examples among them as follows and refer the reader to Section~\ref{sec:example} for a detailed procedure.

\begin{thm:example}
The complement spaces of the Whitehead link, the Borromean ring and a specific 5-chain link in $S^3$ are profinitely rigid in $\mathfrak{M}$ (see Figure~\ref{fig:comb}).
\end{thm:example}
\begin{figure}[h]
  \centering
  \begin{minipage}[bt]{0.3\textwidth}
    \centering
    \includegraphics[width=0.7\textwidth]{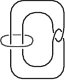}
  \end{minipage} %
  \begin{minipage}[bt]{0.3\textwidth}
    \centering
    \includegraphics[width=0.9\textwidth]{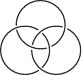}
  \end{minipage}
  \begin{minipage}[bt]{0.3\textwidth}
    \centering
    \includegraphics[width=0.8\textwidth]{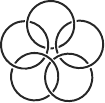}
  \end{minipage} %
  \caption{Three links noted in Theorem~\ref{cor:example}. They are the Whitehead link, the Borreanmean ring, the specific 5-chain link, from left to right respectively.}
  \label{fig:comb}
\end{figure}

The examples in Theorem~\ref{cor:example} follow from a general bubble-drilling construction.
The construction begins from a compact, orientable, fibered 3-manifold $M$ with a fixed fiber structure $\Sigma \to M\to S^1$. An essential simple closed curve $\beta$ in $M$ is called a \textit{bubble} if it lies on a fiber surface.
Given bubbles $\beta_1,\cdots, \beta_n$ on distinct fiber surfaces of $M$, the {\em bubble-drilled manifold} $E= M\setminus \mathring N(\beta_1\cup \cdots \cup \beta_n)$ is obtained by removing the regular neighborhoods of these bubbles.
Under suitable conditions, bubble-drilled manifolds are profinitely rigid in \(\mathfrak{M}\), as established by the following theorem.

\begin{thm:secondmain}

  Let $M$ be a compact, orientable 3-manifold that fibers over $S^1$ with fiber $\Sigma$ being a compact surface, and let $E=M\setminus \mathring N(\beta_1\cup \cdots \cup \beta_n) $ be a bubble-drilled manifold.
If
\begin{itemize}
\item[(HC)]\label{hyperbolicity condition} the interior of \(E\) admits a complete finite-volume hyperbolic structure, and
\item[(RC)]\label{rigidity condition} all fibered hyperbolic 3-manifolds with fiber  \(\Sigma\) are profinitely rigid in $\mathfrak{M}$,
\end{itemize}
then $E$ is profinitely rigid in $\mathfrak{M}$.
\end{thm:secondmain}

On the one hand, the rigidity condition {\RC} is known to hold when \(\Sigma\) is a once-puncture torus by \cite{Br17a} or a four-punctured sphere by \cite{CW23a}. Recently, Wilton-Sisto~\cite{WS24a} provided evidence that the rigidity condition {\RC} may hold for more homeomorphism types of surfaces.

On the other hand, we provide an easily verifiable criterion for the hyperbolicity condition {\HC}.

\begin{thm:hyperbolicity}
The bubble-drilled manifold $E$ is hyperbolic if $M$ is hyperbolic and the bubbles $\beta_i$ are flow-acoannular.
\end{thm:hyperbolicity}

Flow-acoannularity of bubbles is defined as follows.
Let \(\{\widetilde{\beta}_k\}_{k\in\Integer}\) denote all lifts of the \(n\) bubbles \(\beta_i\) in the canonical infinite cyclic covering \(\overline{M}\) of \(M\), corresponding to \(M\to S^1\).
The \(n\) bubbles are called \textit{flow-acoannular} if any two lifts in \(\{\widetilde{\beta}_k\}_{k\in\Integer}\) do not cobound a fiber-transversal annulus that is disjoint from other lifts.
A more computable definition is given in Definition~\ref{def:flow-acoannular 1}.

The proof of Theorem~\ref{thm:secondmain} is based on an observation in~\cite{Xu24b}.
To be precise, \cite{Xu24b} observed that any profinite isomorphism between cusped finite-volume hyperbolic 3-manifolds carries profinite isomorphisms between their Dehn fillings, providing a useful technique for establishing profinite rigidity among hyperbolic 3-manifolds.

\begin{theorem}[{\cite[Theorem A]{Xu24b}}]\label{thm:peripheralDehnfill}
If \(M\) and \(N\) are two profinitely isomorphic finite-volume hyperbolic 3-manifolds with \(k>0\) cusps,
then there exists a homeomorphism $h: \partial M\to \partial N$ such that \(\widehat{\pi_{1}M_{\gamma}}\cong \widehat{\pi_{1}N_{h(\gamma)}}\) for any Dehn filling parameters \(\gamma=(\gamma_{1},\ldots,\gamma_{k})\), where \(\gamma_{i}\) is an empty set or an isotopy class of an essential, oriented simple closed curve on the $i$-th cusp $\partial_iM\cong T^2$,
and \(M_{\gamma}\) is the Dehn filling result along \(\gamma\).
\end{theorem}

\begin{remark}
The notation of the Dehn filling parameter is specified in detail in Section~\ref{subsec:Dehn fill notation}.
With a slight abuse of notation, the boundary components arising from the topological compactification of cusps are also referred to as \textit{cusps}.
\end{remark}

Building on this approach, we are able to determine the profinite rigidity of a cusped finite-volume hyperbolic 3-manifold from a class of known examples via geometric convergence.
The space of all finite-volume hyperbolic 3-manifolds can be endowed with a geometric topology, which measures when hyperbolic 3-manifolds converge quasi-isometrically on increasingly large compact subsets to their geometric limit, with progressively fine quasi-isometry parameters.

\begin{thm:closedset}
The set of all finite-volume hyperbolic 3-manifolds that are profinitely rigid in $\mathfrak{M}$ is closed in the space of all finite-volume hyperbolic 3-manifolds equipped with the geometric topology.
\end{thm:closedset}

A straightforward illustration of Theorem~\ref{thm:closedset} is as follows.  Suppose that $M$ is a finite-volume hyperbolic 3-manifold and $T$ is a cusp of $M$. If infinitely many Dehn fillings of $M$ on $T$ is profinitely rigid in $\mathfrak{M}$, then so is $M$ itself.

\begin{remark}
Suppose that \(N\) is a fibered manifold $N$ with a fixed fiber structure $\Sigma \to N\to S^1$. Following the notation of~\cite{Le21a}, a compact, embedded (possibly disconnected) 1-manifold in $N$ is called \textit{monotonic} if each component is either transverse to the fibration or level, i.e., lying on a fiber surface.
In fact, a compact essential level component is exactly a bubble in $N$ introduced before.

Furthermore, \cite{Le21a} proved that if a sufficiently long Dehn filling $M_\gamma$ of a given cusped hyperbolic manifold $M$ is fibered,
then \(M\) is obtained from a fibered hyperbolic manifold $N$ by removing a monotonic submanifold.
Thus, if a hyperbolic manifold \(M\) is the geometric limit of a sequence of fibered hyperbolic manifolds, then $M$ is either a fibered or a bubbled-drilled manifold, since drilling transverse curves produces a new fibered manifold and preserves the level properties of other components.
However, the fiber type of the resulting bubbled-drilled manifold may not be controlled.
\end{remark}

Based on Theorem~\ref{thm:closedset}, we can reduce the problem of profinite rigidity of finite-volume hyperbolic manifolds to the closed case.
\begin{cor:reducedcase}
Any of the following conditions is sufficient to establish the profinite rigidity of all cusped finite-volume hyperbolic manifolds in $\mathfrak{M}$.

\begin{enumerate}[label=(\arabic*)]
    \item The profinite rigidity of all closed hyperbolic 3-manifolds in $\mathfrak{M}$.
    \item The profinite rigidity of all closed, non-arithmetic hyperbolic 3-manifolds in $\mathfrak{M}$.
    \item The profinite rigidity of all closed, fibered hyperbolic 3-manifolds in $\mathfrak{M}$.
\end{enumerate}
\end{cor:reducedcase}

\subsection*{Organization of the paper}
The organization of this paper is as follows.

In Section~\ref{sec:pre}, we specify our notation about Dehn filling and review profinite completions of finitely generated 3-manifold groups.

In Section~\ref{sec:limit}, we review the geometric topology on the set of finite-volume hyperbolic manifolds and examine the limit behavior of profinite rigidity of finite-volume hyperbolic manifold under geometric topology.

In Section~\ref{sec:fiberedness} and Section~\ref{sec:hyperbolicity}, we propose the bubble-drilling construction and analyze the properties of the resulting manifolds. We prove fiberedness of infinitely many Dehn filling results and establish a hyperbolization criterion.

In Section~\ref{sec:example}, we present some insightful examples of profinitely rigid hyperbolic bubbled-drilled manifold given by bubble construction in detail, which completes the proof of Theorem~\ref{cor:example}.

\subsection*{Acknowledge}
The author sincerely thanks Xiaoyu Xu for helpful conversations on profinite rigidity and valuable advice on a preliminary version of this paper. Also, the author thanks his advisor Yi Liu for helpful comments.

\section{Preliminaries}\label{sec:pre}

\subsection{Profinite completion}\label{sec:pre:prof}

We begin this section with definition of profinite completion.

\begin{definition}
  Let \(G\) be an abstract group.
  The \textit{profinite completion} of \(G\) is defined as the inverse limit of all quotient groups \(G/N\) where \(N\) ranges over all the finite-index normal subgroups of \(G\).
  \begin{equation*}
    \widehat{G}=\underset{N\lhd_{f.i.}G}{\varprojlim} G/N.
  \end{equation*}
\end{definition}

It follows from a theorem of Nikolov-Segal~\cite{NS03a} that, for any two finitely generated groups $G_1$ and $G_2$, any isomorphism between them as abstract groups is indeed an isomorphism as profinite groups—specifically, a homeomorphism with respect to the product topology.
Therefore, in the following context, we do not distinguish between an abstract isomorphism and a continuous one, as we focus exclusively on finitely generated groups.

In a series of works, Wilton-Zalesskii showed that among 3-manifold groups, the profinite completion determines hyperbolicity.
The closed case was proven in~\cite{WZ17a} and the cusped case was proven in~\cite{WZ17b}, see also~\cite[Theorem 4.18 and 4.20]{Re18a}.
Moreover, when $M$ is a cusped hyperbolic manifold, \cite[Proposition 3.1]{WZ19a} showed that the conjugacy classes of peripheral subgroups in $\widehat{\pi_1M}$ are exactly the conjugacy classes of the maximal closed subgroups isomorphic to $\widehat{\mathbb{Z}}^2$.  We conclude these properties in the following proposition.

\begin{proposition}[\cite{WZ17a,WZ17b,WZ19a}]\label{prop:detect hyperbolic}
Let $M$ be a finite-volume hyperbolic 3-manifold. If $N$ is a compact, orientable 3-manifold such that $\widehat{\pi_1M}\cong \widehat{\pi_1N}$, then $N$ is also a finite-volume hyperbolic 3-manifold, and has the same number of cusps as $M$.
\end{proposition}

\subsection{Dehn fillings}\label{subsec:Dehn fill notation}
For a compact, orientable \(3\)-manifold \(M\) with non-empty toral boundaries, we are able to perform many Dehn fillings on its boundary.

We denote the boundaries by \(\partial M=\partial_{1}M \sqcup \cdots \sqcup \partial_{k}M\).
A \textit{Dehn filling parameter} \(\gamma_{i}\) on \(\partial_{i}M\) is either the isotopy class of an essential
simple closed curve in \(\partial_{i}M\) or an empty set, denoted by \(\infty\) simply.
Each simple closed curve \(\gamma_{i}\) on \(\partial_{i}M\) determines a Dehn filling by attaching a solid torus whose meridian matches \(\gamma_{i}\), while \(\infty\) represents leaving this boundary unfilled.
Given a \textit{Dehn filling vector} \(\gamma=(\gamma_{1},\ldots, \gamma_{k})\), let \(M_{\gamma}\) denote the manifold obtained by performing Dehn filling with these parameters.

We view the essential oriented simple closed curves as primitive integral classes in $H_1(\partial_iM;\mathbb{R})\cong \mathbb{R}^2$, and let \(\overline{\mathbb{R}^2}\) be the one-point compactification of the real vector space $H_1(\partial_iM;\mathbb{R})$, where the \(\infty\) parameter exactly corresponds to the infinity point. This induces a well-defined topology on the space of Dehn filling vectors. After specifying a basis for \(H_{1}(\partial_{i}M;\mathbb{Z})\cong \Integer\oplus \Integer\), the Dehn filling coefficient \(\gamma_{i}\) can be parameterized by a coprime couple of integers \((p_{i},q_{i})\) or a $\infty$ symbol.

We recall Thurston’s famous hyperbolic Dehn surgery theorem~\cite[Theorem 5.8.2]{Th80a}, with the following version formulated as in~\cite[Theorem 2.1]{Lac19}.

\begin{proposition}\label{prop: dehn surgery theorem}
    Let $M$ be a cusped finite-volume hyperbolic 3-manifold.
    \begin{enumerate}[label=(\arabic*)]
        \item There exists a neighbourhood $U$ of $(\infty,\ldots,\infty)$ such that, for any Dehn filling vector $\gamma\in U$, $M_\gamma$ is also a finite-volume hyperbolic 3-manifold.
        \item There exists a constant $\epsilon_0=\epsilon_0(M)>0$ such that, for any $0<\epsilon<\epsilon_0$, we can choose a smaller neighbourhood $V_\epsilon$ of $(\infty,\ldots,\infty)$ contained in $U$, which satisfies that $M$ is homeomorphic to the $\epsilon$-thick part of $M_{\gamma}$ for any $\gamma\in V_\epsilon$.
    \end{enumerate}
\end{proposition}

\section{Profinite rigidity and geometric convergence}\label{sec:limit}

\subsection{Geometric topology and geometric convergence}

Geometric topology is introduced in~\cite{Ch50a} first, known as Chabauty topology, and has many equivalent forms. We briefly review one of its versions concerning hyperbolic 3-manifolds, and refer the readers to~\cite{Ca06a} and~\cite[Chapter E]{BP92a} for more backgrounds.

\def\DH{\mathcal{D}_{\ast}(\mathcal{I}^+(\mathbb{H}^3))}

We begin with the definition of the geometric topology on the space of torsion-free Kleinian groups.

\begin{definition}\label{def: equivalent definition for geometric convergence}

  A sequence of torsion-free Kleinian groups \(\{\Gamma_{i}\}_{i=0}^{+\infty}\) \textit{converges geometrically} to a torsion-free Kleinian groups \(\Gamma_{\infty}\) if and only if there exist a sequence of parameters \(\{(R_{i},K_{i})\}\) and a sequence of maps \(F_{i}:B_{R_{i}}(0)\to \Hyperb^{3}\), where \(B_{R_{i}}(0)\) denotes a 3-dimensional ball of radius \(R_i\) centered at a fixed origin \(o\) of \(\Hyperb^{3}\), such that the following conditions hold.
\begin{enumerate}[label=(\arabic*)]
    \item $\lim R_i=+\infty$ and $\lim K_i=1$, as \(i\) tends to \(+\infty\).
    \item Each map \(F_{i}\) is a \(K_{i}\)-bilipschitz diffeomorphism onto its image, and $F_i(o)=o$.
    \item The maps $F_i$ converge locally uniformly to the identity map as $i\to +\infty$.
    \item If let \(M_{\infty}=\Hyperb^{3}/\Gamma_{\infty}\), \(M_{i}=\Hyperb^{3}/\Gamma_{i}\), and \(B_{i}'=B_{R_{i}}(0)/\Gamma_{i}\subset M_{i}\), then each \(F_{i}\) descends to a well-defined map \(f_{i}:B_{i}'\to M_{\infty}\), where \(f_{i}\) is also a \(K_{i}\)-bilipschitz diffeomorphism onto its image.
  \end{enumerate}

  Further, the {\em geometric topology} on the set of torsion-free Kleinian groups is topologized by geometric convergence.
\end{definition}

\def\FVH{\mathcal{H}_3} 

Let $\FVH$ denote the collection of all isometry classes of finite-volume complete hyperbolic 3-manifolds. There is a natural map \(\phi\) from the set of all finite-covolume torsion-free Kleinian group to \(\FVH\), which maps \(\Gamma\) to \(\Hyperb^3/\Gamma\).

\begin{definition}
The {\em geometric topology} on $\FVH$ is the finest topology such that $\phi$ is continuous.
\end{definition}

Strictly speaking, a sequence of finite-volume hyperbolic 3-manifolds \(\{M_{i}\}_{i=0}^{\infty}\subseteq \FVH\) {\em geometrically converges} to \(M\in \FVH\) if there exist torsion-free Kleinian groups \(\{\Gamma_{i}\}_{i=0}^{\infty}\) and $\Gamma$ such that $M_i\cong \mathbb{H}^3/\Gamma_i$, $M\cong \mathbb{H}^3/\Gamma$, and $\Gamma_i$ geometrically converges to $\Gamma$. In other words, by appropriately selecting basepoints, the manifolds $M_i$ and $M$ become increasingly similar within metric balls centered at these basepoints, as both the index $i$ and the radius of the ball tend to infinity.

A typical example of geometric convergence in dimension two is the process of shrinking a non-separating simple closed curve on a closed surface to make its length tend to zero. In the limit state, this results in a surface with two cusps. Likewise, in dimension three, all convergent sequences can be characterized as follows.

\begin{proposition}\label{prop:limit sequence}
   \begin{enumerate}[wide,labelwidth=!, labelindent=0pt]
        \item Let $M$ be a cusped finite-volume hyperbolic 3-manifold, and let $\{\gamma^{(i)}\}_{i=0}^{+\infty}$ be Dehn filling vectors on $M$ converging to $(\infty,\cdots,\infty)$. Then up to deleting finitely many initial terms for where $M_{\gamma^{(i)}}$ might be non-hyperbolic, the remaining terms $\{M_{\gamma^{(i)}}\}_{i=i_0}^{+\infty}$ geometrically converges to $M$ in $\FVH$.
        \item  Let $\{M_i\}_{i=0}^{+\infty}$ be a sequence in $\FVH$ which geometrically converges to $M\in \FVH$. Then up to deleting finitely many initial terms, $M_i$ is a Dehn filling of $M$ with Dehn filling vector $\gamma^{(i)}$, and $\gamma^{(i)}$ converges to $(\infty,\cdots,\infty)$ as $i\to +\infty$.
   \end{enumerate}
\end{proposition}
\begin{proof}
    The first conclusion follows from \cite[Proposition E.6.29]{BP92a}, and the second one is a consequence of \cite[Proposition E.2.4]{BP92a}; see also \cite[Theorem 2.1]{Lac19}.
\end{proof}

In particular, every closed hyperbolic manifold is an isolated point under geometric topology.

\begin{lemma}\label{lem: limit unique}
    Any converging sequence in $\FVH$ has a unique limit. In other words, if $\{M_i\}_{i=0}^{+\infty}\subseteq \FVH$ converges geometrically to two hyperbolic 3-manifolds $M$ and \(N\) simultaneously, then $M$ is homeomorphic, and hence isometric, to $N$.
\end{lemma}
\begin{proof}
It follows from Proposition~\ref{prop:limit sequence}(2) that, after deleting finitely many initial terms in $\{M_i\}$, there exist Dehn filling vectors $\{\gamma^{(i)}\}$ on $M$ and $\{\eta^{(i)}\}$ on $N$, both converging to $(\infty,\cdots, \infty)$, such that $M_{\gamma^{(i)}}\cong M_i\cong N_{\eta^{(i)}}$ for all sufficiently large $i$.
By Proposition~\ref{prop: dehn surgery theorem}(2), there exists a constant $\epsilon>0$, depending on $M$ and $N$, and a sufficiently large index $j$, such that $M$ is homeomorphic to the $\epsilon$-thick part of $M_{\gamma^{(j)}}\cong M_j$, and $N$ is homeomorphic to the $\epsilon$-thick part of $N_{\eta^{(j)}}\cong M_j$.
Therefore, $M$ and $N$ are homeomorphic, since the hyperbolic structure on $M_j$ is unique up to isometry according to the Mostow-Prasad rigidity theorem.
\end{proof}

\subsection{Profinite rigidity of a geometric limit}

\begin{theorem}\label{thm:closedset}

The set of all finite-volume hyperbolic 3-manifolds that are profinitely rigid in $\mathfrak{M}$ is closed in the space of all finite-volume hyperbolic 3-manifolds equipped with the geometric topology.

\end{theorem}
\begin{proof}
  Let \(M\) be a finite-volume hyperbolic manifold and \(\{M_{i}\}{}_{i=0}^{+\infty}\) be a sequence of finite-volume hyperbolic manifolds geometrically converging to \(M\). It suffices to show that if every $M_i$ is profinitely rigid in $\mathfrak M$, so is $M$.

  We may assume $M$ is cusped, because any closed hyperbolic 3-manifold is an isolated point in $\FVH$.
  Therefore, as stated in Proposition~\ref{prop:limit sequence}(2), up to deleting finitely many terms, we may further assume the existence of Dehn filling vectors $\{\beta^{(i)}\}$ on $M$ that converges to $(\infty,\cdots,\infty)$, such that $M_i\cong M_{\gamma^{(i)}}$ for each $i$.

    Let $N$ be a compact, orientable 3-manifold so that $\widehat{\pi_1M}\cong \widehat{\pi_1N}$. Proposition~\ref{prop:detect hyperbolic} implies that $N$ is also cusped finite-volume hyperbolic.
    Therefore, according to Theorem~\ref{thm:peripheralDehnfill}, there exists a homeomorphism $h:\partial M\rightarrow\partial N$ such that \(\widehat{\pi_{1}M_{\gamma}}\cong \widehat{\pi_{1}N_{h(\gamma)}}\) for any Dehn filling vector \(\gamma\) on $M$. In particular, $\widehat{\pi_1M_i}\cong \widehat{\pi_1N_{h(\gamma^{(i)})}} $. Since each $M_i$ is profinitely rigid in $\mathfrak{M}$, it follows that $N_{h(\gamma^{(i)})}\cong M_i$. In particular, $N_{h(\gamma^{(i)})}$ is hyperbolic.

    Since $h$ is a homeomorphism between the boundaries, the vectors $h(\gamma^{(i)})$ also converges to $(\infty,\cdots, \infty)$ as $i$ tends to $+\infty$.
    Consequently, $N$ is the geometric limit of $\{N_{h(\gamma^{(i)})}\}_{i=0}^{+\infty}$ by Proposition~\ref{prop:limit sequence}(1).

    In summary, both \(N\) and \(M\) are the geometric limits of the sequence \(\{M_i\}_{i=0}^{+\infty}\).
    It follows from the uniqueness of the geometric limit (Lemma~\ref{lem: limit unique}) that $N$ is homeomorphic to $M$.
\end{proof}

In the classical language of Dehn fillings, Theorem~\ref{thm:closedset} is restated as in the following corollary.

\begin{corollary}\label{cor: rigidity as dehn vector}
    Let $M$ be a finite-volume hyperbolic 3-manifold with $n$ cusps $\partial_1M,\cdots,\partial_nM$. Let $\{\gamma^{(i)}=(\gamma{}^{(i)}_1,\cdots,\gamma{}^{(i)}_n)\}_{i=0}^{+\infty}$ be a sequence of Dehn filling vectors of $M$, satisfying that, for any $1\le k \le n$ and $i\neq j\ge 0$, either $\gamma{}^{(i)}_k\neq \gamma{}^{(j)}_k$ or $\gamma{}^{(i)}_k= \gamma{}^{(j)}_k=\infty$.
    If each $M_{\gamma^{(i)}}$ is profinitely rigid in $\mathfrak{M}$, then \(M\) is profinitely rigid in $\mathfrak{M}$.

\end{corollary}
\begin{proof}
    It is clear that the conditions in the corollary ensure that the Dehn filling vectors $\gamma^{(i)}$ converges to $(\infty,\cdots,\infty)$ as $i\to +\infty$. Hence, up to deleting finitely many non-hyperbolic terms, $M_{\gamma^{(i)}}$ geometrically converges to $M$ via Proposition~\ref{prop:limit sequence}(1).
    Now, the conclusion follows from Theorem~\ref{thm:closedset} clearly.
\end{proof}

\subsection{Reduce cusped cases to closed cases}

As an application of Theorem~\ref{thm:closedset}, we reduce the profinite rigidity of hyperbolic 3-manifolds to the closed case.

\begin{corollary}\label{cor:reducedcase}
Any of the following conditions is sufficient to establish the profinite rigidity of all cusped finite-volume hyperbolic manifolds in $\mathfrak{M}$.
\begin{enumerate}
    \item All closed hyperbolic 3-manifolds are profinitely rigid in $\mathfrak{M}$.
    \item All closed, non-arithmetic hyperbolic 3-manifolds are profinitely rigid in $\mathfrak{M}$.
    \item All closed, fibered hyperbolic 3-manifolds are profinitely rigid in $\mathfrak{M}$.
\end{enumerate}

\end{corollary}
\begin{proof}
Proposition~\ref{prop:limit sequence}(1) implies that any cusped hyperbolic 3-manifold $M$ is the geometric limit of a sequence of closed hyperbolic Dehn fillings of $M$. Thus, the sufficiency of condition (1) follows from Theorem~\ref{thm:closedset} clearly.

Moreover, there are only finitely many arithmetic terms in this sequence according to \cite[Corollary 11.2.2]{MR13}, that is, $M$ is the geometric limit of a sequence of non-arithmetic closed hyperbolic 3-manifolds, which leads to the sufficiency of condition (2).

To complete the proof, it suffices to show that condition (3) guarantees condition (2). In fact, this is based on a technique recently developed in~\cite{ACM24a}, with a slight modification.

For any closed, non-arithmetic hyperbolic 3-manifold $M$,  \cite[Theorem 4.1]{ACM24a} together with the virtual fibering theorem \cite[Theorem 9.2]{Ago13}  implies that there exists a fibered finite cover ${M}^\star$ of $M$, such that $\pi_1M$ is isomorphic to the normalizer of $\pi_1{M}^\star$ in $\mathrm{PSL}(2,\mathbb{C})$.  Let $N$ be a compact, orientable 3-manifold such that $\widehat{\pi_1M}\cong \widehat{\pi_1N}$. Then $N$ is closed hyperbolic according to Proposition~\ref{prop:detect hyperbolic}. In addition, there exists a finite regular covering ${N}^\star$ of $N$, with $[N^\star:N]=[{M}^\star:M]$ such that $\widehat{\pi_1N^\star}\cong \widehat{\pi_1M^\star}$.
Since $M^\star$ is closed and fibered, it follows that $N^\star$ is homeomorphic to $M^\star$ by condition (3).
On the other hand, according to the Mostow's rigidity theorem, $M$ is the unique hyperbolic 3-manifold regular covered by ${M}^\star$ with the maximal degree of covering.
However, $N^\star\cong M^\star$ and $N^\star\to N$ is a regular cover of the same degree.
Therefore, $N\cong M$, which implies condition (2).
\end{proof}

\section{The bubbling trick}\label{sec:fiberedness}

In this section, we show the fiberedness of infinitely many Dehn fillings of a bubble-drilled manifold, thereby deducing Theorem~\ref{thm:secondmain} from Theorem~\ref{thm:closedset}.

\begin{definition}[bubbled-drilled manifold]
    Let \(M\) be a surface bundle over \(S^1\) with a fixed fiber structure \(\Sigma\to M \to S^1\).
    A \textit{bubble} is an essential simple closed curve on a fiber surface, and a \textit{bubbled-drilled manifold} is defined as a hyperbolic surface bundle removing neighborhoods of finitely many disjoint bubbles on distinct fiber surfaces, usually denoted by \(E\).
\end{definition}

With a slight abuse of notation, the boundary of a bubble-drilled manifold produced by drilling is also called a \textit{bubble}.
Each bubble can be equipped with a canonical meridian-longitude system $(m_i,l_i)$.
To be precise, the {\em distinguished longitude} $l_i$ around a bubble is the normal direction of the bubble in the fiber surface it occupies.
\begin{figure}[t]
\centering
\includegraphics[scale=0.8]{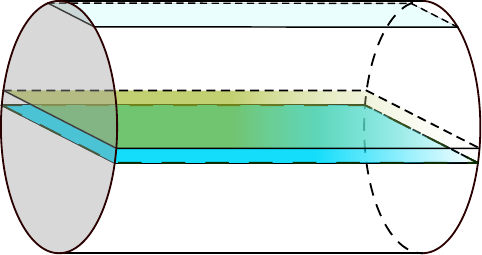}
\caption{Annuli foliation of a solid torus after removing highest and lowest two small arches.}\label{pic:fiberedtorus}
\end{figure}

\begin{proposition}\label{prop:fiberness}
Let $E=M\setminus \mathring N(\beta_1\cup \ldots \cup \beta_n)$ be a bubble-drilled manifold. Then for any integers $k_1,\cdots, k_n$, the Dehn filling of $E$ along the Dehn filling vector $(m_1+k_1l_1,\cdots ,m_n+k_nl_n)$ is also fibered, with the same fiber type as $M$.

\end{proposition}
\begin{proof}
We inherit the fiber structure of $M$ onto $E$, except for the fiber layers that intersect a neighborhood of the bubbles. We may assume that, when encountering each bubble, the ``highest'' and ``lowest'' fiber surfaces intersect with $\partial N(\beta_i)$ as two annuli, with their boundary curves exactly represented by the distinguished longitude $l_i$.

Since the Dehn filling parameter at the bubble is $m_i+k_i l_i$, the distinguished longitude is indeed a longitude in the attached solid torus. Thus, the fiber surfaces can pass through each attached solid torus via an annular foliation of the solid torus shown in Figure~\ref{pic:fiberedtorus}. This yields a fiber structure of the Dehn-filling result with the same fiber type as $M$.
\end{proof}

\begin{example}
    A typical illustration of Proposition~\ref{prop:fiberness} is that integral Dehn surgeries on one component of the Whitehead link are also fibered, with fiber type being a once-punctured torus (see~\cite[Proposition 3]{HMW92}).
\end{example}

\begin{remark}\label{rem:monodromy}
A \(m+pl\) Dehn filling on a bubble \(\alpha\) changes the monodromy of the fiber bundle from \(\varphi\) to \(\varphi \circ D_{\alpha}^{\pm p}\), where \(D_{\alpha}\) denotes the Dehn twist in the fiber \(\Sigma\) along \(\alpha\) and the sign is determined by the orientation of the bubbles' meridian and longitude.
\end{remark}

\begin{theorem}\label{thm:secondmain}
  Let $M$ be a compact, orientable 3-manifold that fibers over $S^1$ with fiber $\Sigma$ being a compact surface and $E=M\setminus \mathring N(\beta_1\cup \cdots \cup \beta_n) $ be a bubble-drilled manifold.
If
\begin{itemize}
\item[(HC)] the interior of \(E\) admits a complete finite-volume hyperbolic structure, and
\item[(RC)] all fibered hyperbolic 3-manifolds with fiber  \(\Sigma\) are profinitely rigid in $\mathfrak{M}$,
\end{itemize}
then $E$ is profinitely rigid in $\mathfrak{M}$.

\end{theorem}
\begin{proof}
  It follows from Proposition~\ref{prop: dehn surgery theorem} that sufficiently long Dehn fillings of \(E\) are hyperbolic. In particular, some of them are hyperbolic fiber bundles over circle with the same fiber type \(\Sigma\) via Proposition~\ref{prop:fiberness}, whose Dehn filling vectors converge to \((\infty,\ldots,\infty)\). These Dehn fillings are profinitely rigid in $\mathfrak M$ according to   condition (RC), so $E$ is profinitely rigid as well via Corollary~\ref{cor: rigidity as dehn vector}.
\end{proof}

\section{The hyperbolicity criterion}\label{sec:hyperbolicity}
In this section, we prove the hyperbolicity of drilled manifolds with \(n\) bubbles under   easily checkable conditions.

Suppose that \(M=\Sigma\times I/\sim\), where \((x,1)\sim(\varphi(x),0)\), is a hyperbolic surface bundle over \(S^{1}\) with possibly non-empty toral boundaries. It is clear that the monodromy \(\varphi\) is pseudo-Anosov by Thurston's geometrization theorem and there is a canonical infinite cyclic covering \(\pi_{\Integer}:\Sigma\times \Real\to M\) with \(\Integer\)-action \(\widetilde{\varphi}(x,t)=(\varphi(x),t-1)\).

We compact our condition as follows.

\begin{definition}[upright annulus]\label{def:upright}
    An annulus in \(\Sigma\times \Real\) is called \textit{upright} if its interior is transverse to each level \(\Sigma\times \{t\}\) and its two circle boundaries belong to different levels.
\end{definition}
\begin{definition}[flow-acoannular, finite sequence version]\label{def:flow-acoannular 1}
  Let \(\{\alpha_{i}\}_{1\leq i\leq n}\) be a finite sequence \(\mathcal{A}_{n}\) of essential closed curves on \(\Sigma\), some of which are possibly  isotopic.
  This finite sequence can be extended to an infinite one \(\mathcal{A}_{\Integer\varphi}\) by setting \(\alpha_{i+kn}=\varphi^{-k}(\alpha_{i}), k\in \Integer\).

  The finite sequence \(\mathcal{A}_{n}\) is said to be \textit{flow-acoannular} if for any two terms \(\alpha_{k}\) and \(\alpha_{l}\), \(k\neq l\), in the extended sequence \(\mathcal{A}_{\Integer\varphi}\) that are isotopic,  there exists \(k<h<l\), which satisfies that \(\alpha_{h}\) has positive geometric intersection number with \(\alpha_{k}\).
\end{definition}
\begin{definition}[flow-acoannular, bubble version]\label{def:flow-acoannular 2}
  Let \(\beta_{i}=\alpha_{i}\times \left\{ t_{i} \right\}\) be $n$ bubbles in \(M\), where \(0<t_{1}<t_{2}< \cdots <t_{n}<1\) are \(n\) real numbers that describe the heights of the bubbles. They are called \textit{flow-acoannular} if their projections \(\{\alpha_{i}\}_{1\leq i\leq n}\) onto \(\Sigma\) form a flow-acoannular finite sequence as above.

Equivalently, let \(\{\widetilde{\beta}_k\}_{k\in\Integer}\) denote all lifts of bubbles \(\beta_i\) in the canonical infinite cyclic covering \(\Sigma\times \Real\).
The \(n\) bubbles are called \textit{flow-acoannular} if any two lifts \(\widetilde{\beta}_{k_1}\) and \(\widetilde{\beta}_{k_2}\) do not cobound an upright annulus disjoint from other lifts.
\end{definition}

\begin{theorem}\label{thm:hyperbolicity}
If \(M\) is hyperbolic and the \(n\) bubbles \(\{\beta_i\}_{i=1}^n\) are flow-acoannular, then the bubble-drilled manifold \(E\) is hyperbolic.
\end{theorem}

Combining the Hyperbolic Dehn filling theory and the implication on the monodromy (see Remark~\ref{rem:monodromy}), Theorem~\ref{thm:hyperbolicity} implies the following corollary.

\begin{corollary}
Let \(\varphi\in \MCG(\Sigma)\) be a pseudo-Anosov mapping class and \(\{\alpha_{i}\}_{1\leq i\leq n}\) be a finite sequence of essential simple closed curves, some of which are possibly isotopic. If \(\{\alpha_{i}\}_{1\leq i\leq n}\) forms a flow-acoannular finite sequence, then there exists a family of positive bounds \(K_{i}\), such that for any parameters \(|k_{i}|>K_{i}\), \(\varphi\circ D_{\alpha_{1}}^{k_{i}}\circ \cdots \circ D_{\alpha_{n}}^{k_{n}}\) is also a pseudo-Anosov mapping class.
\end{corollary}

The rest of this section is devoted to the proof of Theorem~\ref{thm:hyperbolicity}. We first specify our notation and establish some lemmas. The goal of these lemmas is to demonstrate that many behaviors of bubbles resemble those of real geodesics in \(M\).

Fix a hyperbolic structure on \(M\) via a covering map \(\pi_{M}:\Hyperb^{3}\to M\), and a hyperbolic structure on \(\Sigma\) by a covering map \(\pi_{\Sigma}:\Hyperb^{2}\to \Sigma\). There are two useful metrics on the universal covering \(\widetilde{M}\) of \(M\): one is \(\Hyperb^3\) which is induced by \(\pi_M\); and the other is \(\Hyperb^2\times \Real\), the product metric of the hyperbolic metric on \(\widetilde{\Sigma}\) and the standard metric of \(\Real\). Although the \(\Hyperb^2\times \Real\) metric is not invariant under the deck transformation group of \(\widetilde{M}\to M\), lifts of bubbles exhibit elegant geometry in this metric, whereas they display pathological geometry in the \(\Hyperb^3\) metrics.

Without loss of generality, when referring to the lift of bubbles in \(\widetilde{M}\), we always assume that it is a (level) geodesic in the \(\Hyperb^2\times \Real\) metric.

\begin{lemma}\label{lem:one bubble lift exterior}
  For any essential simple closed curve \(\alpha\in \Sigma\times \left\{ t \right\}\) and any lift \(\widetilde{\alpha}\) of \(\alpha\) in universal covering \(\widetilde{M}\) of \(M\).
  The complement of \(\widetilde{\alpha}\) is homeomorphic to the interior of a solid torus. In particular, \(\pi_{1}(\widetilde{M}\setminus \widetilde{\alpha})\cong \Integer\).
\end{lemma}
\begin{proof}
\(\widetilde{\alpha}\) is isotopic to a straight, level hyperbolic line on \(\Hyperb^2\times \{t\}\) in the \(\Hyperb^{2}\times \Real\) metric of \(\widetilde{M}\).
The conclusion follows straightforwardly.
\end{proof}

\begin{lemma}\label{lem:different ends}
Let $\alpha \in \Sigma\times \{t\}$ be an essential simple closed curve, and let $\widetilde{\alpha}$ be a lift in $\widetilde{M}\cong \mathbb{H}^3$. Then $\widetilde{\alpha}$ is a quasi-geodesic. In particular, the two ends of $\alpha$ in $\partial_{\infty} \mathbb{H}^3$ are disjoint.
\end{lemma}

\begin{proof}
We view $\pi_1(M)$ as an isometry group acting on $\mathbb{H}^3$.
Since $\pi_1(M)=\pi_1(\Sigma)\ast_{\varphi}$, the element
\([\alpha]\in\pi_1(M)\) is not conjugate into a peripheral subgroup, and is hence loxodromic. Thus, $\widetilde{\alpha}$ is a quasi-geodesic, whose ends coincide with those of the axis preserved by $[\alpha]$.
\end{proof}

The following two observations are crucial in the proof of Theorem~\ref{thm:hyperbolicity}.
\begin{lemma}\label{lem:lifts no common ends}
Different lifts \(\widetilde{\alpha}_i\) of a fixed essential simple closed curve \(\alpha\subset \Sigma\times \left\{ t \right\}\) have disjoint ends on the ideal boundary.
\end{lemma}
\begin{proof}

Let $\widetilde{\alpha}$ be one lift of $\alpha$, and any lift of $\alpha$ can be described as $g(\widetilde{\alpha})$, where $g\in \pi_1M$.
Let $\gamma=[\alpha]\in \pi_1(M)$ be the conjugacy representative corresponding to this lift. Then
 $g(\widetilde{\alpha})$ and $\widetilde{\alpha}$ represent the same lift if and only if $g\in \left\langle \gamma\right\rangle$.

The ends of $\widetilde{\alpha}$ and $g\widetilde{\alpha }$ are exactly, respectively, the fixed points of the loxodromic elements $\gamma$ and $g\gamma g^{-1}$ in $S^2_{\infty}$. If they have one common end, then they actually have two common ends since the isometric action of $\pi_1(M)$ on $\mathbb{H}^3$ is discrete. Moreover, $g$ belongs to the elementary subgroup $E(\gamma)$ associated with $\gamma$. Note that $\pi_1M$ is torsion-free, so $E(\gamma)\cong \mathbb{Z}$, and is generated by the primitive loxodromic element in this class.

Observe that the element $\gamma$ is primitive in $\pi_1(\Sigma)$, and $\pi_1(M)=\pi_1(\Sigma)\rtimes_\varphi \mathbb{Z}$, so $\gamma$ is also primitive in $\pi_1(M)$. Thus, $E(\gamma)=\left\langle \gamma\right\rangle$, and $g\in \left\langle \gamma\right\rangle$.
\end{proof}

\begin{lemma}\label{lem:parallel}
  Let \(\beta_1=\alpha_1\times\{t_1\}\) and \(\beta_2=\alpha_2\times\{t_2\}\) be two bubbles in \(M\), where \(\alpha_1\) and $\alpha_2$ are two essential simple closed curves in \(\Sigma\) and \(t_1\neq t_2\in(0,1)\).
  The following are equivalent:
  \begin{enumerate}
      \item \(\beta_1\) and \(\beta_2\) are freely homotopic in \(M\).
      \item There exists \(k\in \Integer\) such that \(\varphi^{k}(\alpha_1)\) is isotopic to \(\alpha_2\).
      \item Given any lift \(\overline{\beta}_1\) of \(\beta_1\) in \(\Sigma\times \Real\), there exists a lift of \(\beta_2\) in \(\Sigma\times \Real\), which cobounds an upright annulus together with \(\overline{\beta}_1\) (Definition~\ref{def:upright}).
      \item Given any lift \(\widetilde{\beta}_1\) of \(\beta_1\) in the universal covering \(\widetilde{M}\cong \Hyperb^{3}\), there exists a lift of \(\beta_2\) in \(\widetilde{M}\), which has the same accumulative ends as \(\widetilde{\beta}_1\) at the ideal boundary \(S^2_\infty\).
  \end{enumerate}
\end{lemma}
\begin{proof}
    \((1)\Rightarrow(2)\) is clear from that \(\pi_1(M)\) is a HNN-extension \(\pi_1(\Sigma)*_\varphi\) of
    \(\pi_1(\Sigma)\) and \((2)\Rightarrow(3)\) and \((3)\Rightarrow(4)\) are straightforward.
    Thus, it suffices to show that \((4)\Rightarrow(1)\).
    If two respective lifts of \(\beta_1\) and \(\beta_2\) in \(\widetilde{M}\) have the same accumulative ends at the ideal boundary \(S^2_\infty\), it follows that \([\beta_1]\) and \([\beta_2]\) belong to the same conjugacy class in \(\pi_1(M)\) since \([\beta_i]\in\pi_1(M)\) is primitive and loxodromic.
    The conclusion follows from the bijection between the conjugacy classes in \(\pi_1(M)\) and the freely homotopic equivalence class of closed curves in \(M\).
\end{proof}

\begin{remark}
  It is worth noting that the pathological geometric images of bubble lifts in the universal covering \(\Hyperb^{3}\) are clear by the work on the Cannon-Thurston map.
  Specifically speaking, Bowditch's results in~\cite[Section 9]{Bo07a} showed the following surprising fact.
  \begin{fact}
    The lift \(\Hyperb^{2}\to\Hyperb^{3}\) of \(\Sigma\to M\) uniquely extends to a continuous, surjective map \(\iota:\Sph^{1}\cong \partial \Hyperb^{2}\xrightarrow{\iota} \partial \Hyperb^{3}\cong \Sph^{2}\).
    Moreover, the endpoints of any leaf of the two invariant laminations \(\lambda^{\pm}\) (one stable and the other unstable) have the same image under \(\tau\), and in fact this case generates all identifications occurring under the surjective map \(\tau\).
  \end{fact}
\end{remark}
Now we turn to the proof of Theorem~\ref{thm:hyperbolicity}. In fact, according to Sakai's observation and the proof formulated in~\cite{Ko88a}, a manifold obtained by drilling a simple geodesic link in a closed hyperbolic manifold is hyperbolic. However, our case is more complicated due to the pathology of bubbles.
\begin{proof}[Proof of Theorem~\ref{thm:hyperbolicity}]
  By Thurston's uniformization theorem for manifold with boundary~\cite{Th82a,Mo84a}, we will show one by one that the bubble-drilled manifold \(E\) is irreducible, admits no Seifert fibrations, boundary-incompressible and contains no essential, non-boundary-parallel tori.
\begin{enumerate}[label=(\arabic*),nosep,labelindent=\parindent, leftmargin=0pt, widest=0, itemindent=*]
\item \(E\) is irreducible.
  Since all \(\beta_{i}\) are homotopically nontrivial, they cannot lie in a $3$-ball. The irreducibility of \(E\) follows from the irreducibility of \(M\).
\item \(E\) admits no Seifert fibrations.
  Since there is one specific Dehn filling \(E\) that restores \(M\), it follows that if \(E\) has a Seifert fibration, then \(M\) either is reducible or admits a new Seifert fibration. Both contradict the assumption that \(M\) is hyperbolic.
\item \(E\) is boundary-incompressible.
  Otherwise, by irreducibility, $E$ is homeomorphic to a solid torus, which is impossible.
\item \(E\) contains no incompressible, non-boundary-parallel tori.
  If there exists one, denoted by \(T\), then \(T\) is compressible in \(M\).
  By Dehn's lemma,
  there exists a properly embedded 2-disk $(D,\partial D)\hookrightarrow (M,T)$, and clearly there exists a bubble \(\beta_i\) that intersects $D$. Then, \((T\setminus N(\partial D) ) \cup (D\times \left\{ -\epsilon,\epsilon \right\} ) \cong S^{2}\) bounds a 3-ball $B$ in $M$ by irreducibility.

  Now we choose lifts \(\widetilde{T}\) and \((\widetilde{D} ,\partial \widetilde{D} )\hookrightarrow(\widetilde{M},\widetilde{T})\) in the universal covering of \(M\).
  \(\widetilde{T}\) may be homeomorphic to a torus or an open annulus since \(\pi_{1}(T)\xrightarrow{i_{*}} \pi_{1}(M) \) has non-trivial kernel. We discuss them respectively.
  \end{enumerate}
\begin{enumerate}[label=(\Roman*),nosep,labelindent=\parindent, leftmargin=0pt, widest=0, itemindent=*]
  \item If \(\widetilde{T}\) is homeomorphic to a torus.
  Select a cylinder \(B(0,r)\times (-t,t)\subset \widetilde{M}\cong \Hyperb^{2}\times \Real\), where \(r,t\) are sufficiently large so that its interior contains \(\widetilde{T}\). Then \(\widetilde{T}\subset B(0,r)\times (-t,t) -\cup_{l}\widetilde{\beta}^{(l)}\) where \(\widetilde{\beta}^{(l)}\) are finitely many lifts of the bubbles that intersect the cylinder.
  However, the latter manifold is a cylinder drilling finitely many disjoint straight lines on levels, and is hence homeomorphic to  the interior of a handlebody. In particular, it contains no closed incompressible surfaces and we have obtained a contradiction.

  \item If \(\widetilde{T}\) is homeomorphic to an open annulus. Since \(\widetilde{T}\) is invariant by the deck transformation of some prmitive element \(\gamma\in \pi_{1}(M)\), each end of \(\widetilde{T}\) approaches a limit point of $\gamma$ in the ideal boundary \(S^{2}_{\infty}\). Let \(\widetilde{U}\) be the closure of the component of \(\Hyperb^{3}-\widetilde{T}\) whose ends meet \(S^{2}_{\infty}\) in exactly these points.
  Actually, we can also lift \(D\) and \(B\) to \(\widetilde{U}\subseteq \Hyperb^{3}\).
  Thus, the projection \(\widetilde{U}\to \pi_{M}(\widetilde{U})=U\) is an infinite cyclic covering,   \(\widetilde{U}\) is a  bi-infinite sequence of 3-balls \(\widetilde{B}_{k}\) concated by 2-disks \(\widetilde{D}_{k}\). In particular, $\widetilde{U}$ is homeomorphic to a closed 3-ball removing the north and south poles.

 Recall that $\widetilde{D}\subseteq \widetilde{U}$.
  Since $\widetilde{D}$ intersects a lift $\widetilde{\beta_i}$ of a bubble $\beta_i\subseteq M\setminus T$, it follows that $\widetilde{U}\supseteq \beta_i$. In other words, $\widetilde{U}$ contains a quasi-geodesic by Lemma~\ref{lem:different ends}, so $\gamma$ is a loxodromic element, and $\widetilde{U}$ has two disjoint ends.
  We divide the discussion into two cases based on the number of lifts of bubbles contained in $\widetilde{U}$.

  \end{enumerate}
\begin{enumerate}[label=(\roman*),nosep,labelindent=\parindent, leftmargin=0pt, widest=0, itemindent=*]
    \item Assume that there is only one lift \(\widetilde{\beta}\) in \(\widetilde{U}\). We split \(\Hyperb^{3}-\widetilde{\beta}\) by \(\widetilde{T}\) into two parts.
  Since \(\widetilde{T}\) is incompressible both in \(\Hyperb^{3}-\Int\widetilde{U}\) and \(\widetilde{U}-\widetilde{\beta}\),
  \(\Integer\cong\pi_{1}(\Hyperb^{3}-\widetilde{\beta})\) is isomorphic to the amalgamated product of \(\pi_{1}(\Hyperb^{3}-\Int\widetilde{U})\) and \(\pi_{1}(\widetilde{U}-\widetilde{\beta})\) over \(\pi_{1}(\widetilde{T})\cong \Integer\) via van Kampen's theorem and Lemma~\ref{lem:one bubble lift exterior}.
  Hence, \(\pi_{1}(\widetilde{U}-\widetilde{\beta})\) must be isomorphic to \(\Integer\).

  It is clear that $\widetilde{U}-\widetilde{\beta}=\pi_M^{-1}(U\cap E)\cap \widetilde{U}$  is a  \(\Integer\)-covering space of \(U\cap E\), whose fundamental group is isomorphic to \(\Integer\).
  By the fibration theorem of Stallings~\cite{St62a}, \(U\cap E\) fibers over the circle so that the fiber \(F\) has fundamental group \(\pi_{1}(F)\cong \Integer\). \(F\) must be an annulus and \(U\cap E\) is homeomorphic to \(T^{2}\times I\) because \(M\) is orientable. Therefore, \(T\) is parallel to a boundary of component of \(E\), leading to a contradiction.

  \item Assume that there are more than one lifts of bubbles in \(\widetilde{U}\).

Suppose $\widetilde{\beta}_1$ and $\widetilde{\beta}_2$ are two lifts of bubbles \(\beta_1,\beta_2\), respectively and are contained in $\widetilde{U}$.
Then the ends of  $\widetilde{\beta}_1$ and $\widetilde{\beta}_2$ at \(S^2_\infty\) are identical with the ends of $\widetilde{U}$.
It follows from Lemma~\ref{lem:lifts no common ends} that \(\beta_1\neq\beta_2\).
Let \(\widetilde{\pi}_\Sigma\) denote the canonical covering map \(\widetilde{M}\cong \Hyperb^2\times \Real\to\Sigma\times \Real\).
Hence, by Lemma~\ref{lem:parallel},
there exist a lift $\widetilde{\beta}_2'$ of \(\beta_2\) such that $\widetilde{\beta}_2'$ and $\widetilde{\beta}_1$ have identical ends at \(S^2_\infty\), and  \(\widetilde{\pi}_\Sigma(\widetilde{\beta}_1)\) and \(\widetilde{\pi}_\Sigma(\widetilde{\beta}_2')\) cobound an upright annulus $H$ in $\Sigma\times\mathbb{R}$.
Consequently, \(\widetilde{\beta}_2'=\widetilde{\beta}_2\) via Lemma~\ref{lem:lifts no common ends} again.
Now, flow-acoannularity of bubbles implies that there exists a level $\Sigma\times \{t_0\}$ which intersects the interior of $H$ and contains a bubble lift $\varrho$, such that $H\cap (\Sigma\times \{t_0\})$ has positive geometric intersection number with $\varrho$ in $\Sigma\times \{t_0\}$.

As before, when referring to a bubble lift in \(\widetilde{M}\) or \(\Sigma\times \Real\), we always assume that it has been isotoped to be a geodesic in the \(\Hyperb^2\times \Real\) metric.
  Under this assumption, the upright annulus between two lifts can be chosen with the form \(\alpha_{*}\times [t_{1},t_{2}]\).
In summary, in $\widetilde{M}\cong\mathbb{H}^2\times \mathbb{R}$, there exists a lift $\widetilde{H}\cong \mathbb{R}\times [0,1]$ of \(H\) such that $\partial \widetilde{H}=\widetilde{\beta}_1\cup \widetilde{\beta}_2$, and
further, a lift $\widetilde\varrho$ of \(\varrho\) that intersects transversely with $\widetilde{H}$ exactly once.

Now we add in the ideal boundary $S^2_\infty$ of $\mathbb{H}^3\cong \widetilde{M}$. We claim that the closure of $\widetilde{\varrho}$ and $\widetilde{U}$ in $\overline{\mathbb{H}^3}$ are disjoint.
Otherwise, \(\widetilde{U}\), $\widetilde{\varrho}$, $\widetilde{\beta}_1$ and $\widetilde{\beta}_2$ have common ends since $\widetilde{\varrho}$ is disjoint with $\partial \widetilde{U}=\widetilde{T}$.
Therefore, according to Lemma~\ref{lem:parallel}, $\varrho=\widetilde{\pi}_\Sigma(\widetilde{\varrho})$ is freely homotopic to $\widetilde{\pi}_\Sigma(\widetilde{\beta})$ in $\Sigma\times \mathbb{R}$, which contradicts our construction.

\def\cls{\mathrm{cls}}
In addition, the closure of $\widetilde{U}$ in $\overline{\mathbb{H}^3}$ is homeomorphic to a closed 3-ball and
$\sigma =\cls(\widetilde{\beta}\cup \widetilde{\beta}')$ forms a loop in $\cls(\widetilde{U})\subseteq \overline{\mathbb{H}^3}-\cls(\widetilde{\varrho})$.
Since $\cls(\widetilde{U})$ is contractible, it follows that $\sigma$ is null-homotopic in $\overline{\mathbb{H}^3}-\cls(\widetilde{\varrho})$.
However, it is clear that $(\overline{\mathbb{H}^3},\cls(\widetilde{\varrho}))$ is homeomorphic to the standard pair $(\overline{B^3}, \text{diameter})$, and $\sigma$ bounds a 2-disk $\cls(\widetilde{H})$ in $\overline{\mathbb{H}^3}$, which intersects transversly with $\widetilde{\varrho}$ exactly once.
Therefore, $\sigma$ is not null-homotopic in  $\overline{\mathbb{H}^3}-\cls(\widetilde{\varrho})$, which leads to a contradiction.
\end{enumerate}
\end{proof}

\begin{figure}[b]
  \centering
  \begin{minipage}[bt]{0.42\textwidth}
    \centering
\includegraphics[width=0.8\linewidth]{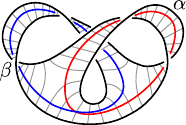}
    \captionsetup{width=0.8\textwidth}
    \caption{A fiber of the fibration of figure-right knot}\label{fig:Seifertsurface}
  \end{minipage} %
  \begin{minipage}[bt]{0.5\textwidth}
    \centering
    \includegraphics[width=\linewidth]{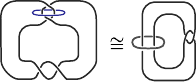}
    \captionsetup{width=\textwidth}
    \caption{A bubbled-drilled manifold homeomorphic to Whitehead link complement}\label{fig:result-whiteheadlink}
  \end{minipage}

\end{figure}

\section{New examples for profinite rigidity}\label{sec:example}

According to Theorem~\ref{thm:secondmain}, we can provide numerous examples of profinitely rigid 3-manifolds, derived from the easily verified hyperbolization condition (see Theorem~\ref{thm:hyperbolicity}) and the well-known profinitely rigid surface bundles over \(S^{1}\), namely, once-punctured torus~\cite{Br17a} and four-punctured sphere~\cite{CW23a}.
Many of these examples exhibit highly complicated hyperbolic structures.
For example, we are able to construct profinitely rigid finite-volume hyperbolic 3-manifolds with arbitrarily many cusps.
In this section, we present several easily understandable examples among them.

If the origin manifold \(M\), before being bubble-drilled, is the complement of a fibered knot or a fibered link in \(S^{3}\),
then bubbled-drilled manifold \(E\) is a link complement in \(S^{3}\), which can be drawn on a sheet conveniently.
Following this procedure, we can prove that many link complements in \(S^{3}\) are profinitely rigid.
We select some particularly elegant examples and exhibit the procedure in detail, which proves the following theorem.

\begin{theorem}\label{cor:example}
The complement spaces of the Whitehead link, the Borromean ring and a specific 5-chain link in $S^3$ are profinitely rigid in $\mathfrak{M}$ (see Figure~\ref{fig:comb}).
\end{theorem}

\paragraph{\textbf{(1) The origin manifold $M$ is the figure-eight knot complement.}}
It is well known that the figure-eight knot complement is a hyperbolic once-punctured torus bundle over \(S^{1}\), whose fiber surface  is illustrated in Figure~\ref{fig:Seifertsurface}.
In this case, the rigidity condition {\RC} in Theorem~\ref{thm:secondmain} is satisfied according to~\cite{Br17a}.

The fundamental group of the fiber is generated by two simple closed curves \(\alpha\) (in red) and \(\beta\) (in blue), with a geometric intersection number of one.
For instance, drilling \(\beta\) away produces a bubbled-drilled manifold homeomorphic to the Whitehead link complement (see Figure~\ref{fig:result-whiteheadlink}). Drilling both \(\alpha\) and \(\beta\) away in specified adjacent fibers simultaneously results in a bubble-drilled manifold homeomorphic to the Borromean ring complement (see Figure~\ref{fig:result-Borromeanring}).
In these cases, the hyperbolicity condition {\HC} is satisfied according to~\cite[Chapter 3.3 \& 3.4]{Th80a}.
Alternatively, one can also show that these bubbles form a flow-acoannular sequence, so as to obtain condition {\HC} from Theorem~\ref{thm:hyperbolicity}.
Therefore, the Whitehead link complement and the Borromean ring complement are both profinitely rigid in $\mathfrak{M}$.

\begin{figure}[t]
\centerline{\includegraphics[width=0.8\textwidth]{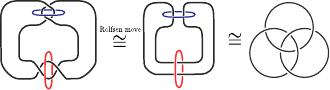}}
\caption{A bubbled-drilled manifold homeomorphic to Borromean ring complement}\label{fig:result-Borromeanring}
\end{figure}

\begin{figure}[b]
  \centering
  \begin{minipage}[bt]{0.48\textwidth}
    \centering
    \includegraphics[width=0.6\textwidth]{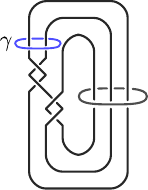}
    \captionsetup{width=0.95\textwidth}
    \caption{A specific hyperbolic three-punctured disk surface bundle over \(S^{1}\). \(\gamma\) is a chosen bubble.}\label{fig:3-braid}
  \end{minipage} %
  \begin{minipage}[bt]{0.48\textwidth}
    \centering
    \includegraphics[width=0.5\textwidth]{5-chain.pdf}
    \captionsetup{width=\textwidth}
    \caption{The drilled-manifold illutrated left is homeomorphic to a 5-chain link complement.}\label{fig:5-chain}
  \end{minipage}

\end{figure}

\paragraph{\textbf{(2) The origin manifold $M$ is a  three-punctured disk bundle over \(S^{1}\)}}
Three-punctured disk bundles are  special examples  of four-punctured sphere  bundles over \(S^{1}\) and, as such, profinitely rigid in $\mathfrak{M}$ according to~\cite{CW23a}. Thus, the rigidity condition {\RC} holds in this case.

There is a natural model for these manifolds \(M\) as a link complement in \(S^{3}\).
To construct the link, one can simply close a 3-braid while passing it through an unknot.
The fiber can then be chosen to embed into the disk bounded by the unknot.
By repeating the bubble-drilling construction
  resembling those outlined earlier, we generate a large number of profinitely rigid examples. We select an insightful one among them.

For a specific three-punctured disk bundle over \(S^{1}\) (see Figure~\ref{fig:3-braid}), we drill the bubble \(\gamma\), perform some isotopies and Rolfsen moves and ultimately obtain a 5-chain link (see Figure~\ref{fig:5-chain}). Note that the complement of any 5-chain link is hyperbolic according to~\cite[Theorem  5.1(ii)]{MR13}, so the hyperbolicity  condition {\HC} holds. Alternatively, one can show that this three-punctured disk bundle is hyperbolic, and deduce condition {\HC} from Theorem~\ref{thm:hyperbolicity}, since flow-acoannularity automatically holds when there is exactly one bubble. Thus, the complement of this 5-chain link is profinitely rigid in $\mathfrak{M}$.

\bibliographystyle{alpha}
\bibliography{bib.bib}
\end{document}